%% file: Main.tex
\documentclass{amsart}
\pdfoutput=1
\usepackage{amssymb, latexsym, amsmath, verbatim, amsthm, amscd, textcomp}
\usepackage{pinlabel}
\usepackage{graphicx}
\usepackage{hyperref}

\def\mytitle{A Geometric Proof of the Structure Theorem for Cyclic
  Splittings of Free Groups}
\def\mykeywords{free group, virtually free group, group splitting,
  cyclic splittings of free groups}
\title{\mytitle}
\author{Christopher H. Cashen} 
\address{Christopher H. Cashen\newline \indent
Fakult\"at f\"ur Mathematik\newline \indent
Universit\"at Wien\newline \indent
1090 Vienna\\Austria}
\email{\href{mailto:cashenchris@gmail.com}{cashenchris@gmail.com}}
\urladdr{\href{http://www.mat.univie.ac.at/~cashen}{http://www.mat.univie.ac.at/~cashen}}
\date{\today}

\thanks{I gratefully acknowledge support from grant
  ANR-2010-BLAN-116-01 GGAA and the European Research Council (ERC) grant of Goulnara
ARZHANTSEVA, grant agreement \textnumero 259527.}

\hypersetup{
    pdftitle={\mytitle},    % title
    pdfauthor={Christopher H Cashen},     % author
    pdfkeywords={\mykeywords}, % list of keywords
    colorlinks=true,       % false: boxed links; true: colored links
    linkcolor=black,          % color of internal links
    citecolor=black,        % color of links to bibliography
    filecolor=black,      % color of file links
    urlcolor=black           % color of external links
}

\theoremstyle{plain}
\newtheorem{theorem}{Theorem}[section]
\newtheorem{lemma}{Lemma}[section]
\newtheorem{proposition}{Proposition}[section]
\newtheorem{corollary}{Corollary}[section]

\theoremstyle{remark}

\newtheorem*{remark}{Remark}

\theoremstyle{definition}
\newtheorem{definition}{Definition}[section]
\newtheorem{example}{Example}[section]

\makeatletter
\@addtoreset{claim}{proposition}
\makeatother
\makeatletter
\@addtoreset{claim}{lemma}
\makeatother

\def\makeautorefname#1#2{\expandafter\def\csname#1autorefname\endcsname{#2}}
\let\fullref\autoref

\makeautorefname{virtualsplittingtheorem}{Virtual Splitting Theorem} 
\makeautorefname{theorem}{Theorem} 
\makeautorefname{lemma}{Lemma} 
\makeautorefname{proposition}{Proposition} 
\makeautorefname{corollary}{Corollary} 
\makeautorefname{definition}{Definition}
\makeautorefname{example}{Example}
\makeautorefname{section}{Section}
\makeautorefname{subsection}{Section}
\makeautorefname{subsubsection}{Section}
\makeautorefname{claim}{Claim}

\makeatletter 
\let\c@lemma=\c@theorem 
\makeatother
\makeatletter 
\let\c@proposition=\c@theorem 
\makeatother
\makeatletter 
\let\c@corollary=\c@theorem 
\makeatother
\makeatletter 
\let\c@definition=\c@theorem 
\makeatother
\makeatletter 
\let\c@example=\c@theorem 
\makeatother

\DeclareMathOperator{\Aut}{Aut} % automorphism group
\DeclareMathOperator{\Comm}{Comm}% commensurator
\def\bdry{\partial}
\def\X{\mathcal{X}} 

\def\tree{\mathcal{T}} % tree

\def\2cc{2--connected components}

\makeatletter
\newsavebox\myboxA
\newsavebox\myboxB
\newlength\mylenA

\newcommand*\xoverline[2][0.75]{%
    \sbox{\myboxA}{$\m@th#2$}%
    \setbox\myboxB\null% Phantom box
    \ht\myboxB=\ht\myboxA%
    \dp\myboxB=\dp\myboxA%
    \wd\myboxB=#1\wd\myboxA% Scale phantom
    \sbox\myboxB{$\m@th\overline{\copy\myboxB}$}%  Overlined phantom
    \setlength\mylenA{\the\wd\myboxA}%   calc width diff
    \addtolength\mylenA{-\the\wd\myboxB}%
    \ifdim\wd\myboxB<\wd\myboxA%
       \rlap{\hskip 0.5\mylenA\usebox\myboxB}{\usebox\myboxA}%
    \else
        \hskip -0.5\mylenA\rlap{\usebox\myboxA}{\hskip 0.5\mylenA\usebox\myboxB}%
    \fi}
\makeatother

\newcommand{\inv}[1]{\xoverline{#1}}

\DeclareMathOperator{\Wh}{Wh}
\DeclareMathOperator{\rank}{rank}
\newcommand{\multimot}{\underline{w}}

\newcommand{\closure}[1]{\overline{#1}} % closure

\newlength{\figstandardheight}
\def\from{\colon\thinspace}

\begin{document}

\input{abstract}

\maketitle

%\tableofcontents

%%%%%%%%%%%%%%%%%%%%%%%%%%%%%%%%%%%%
% End of frontmatter
%%%%%%%%%%%%%%%%%%%%%%%%%%%%%%%%%%%%

%%%%%%%%%%%%%%%%%%%%%%%%%%%%%%%%%%%%%
% Start of main body of article
%%%%%%%%%%%%%%%%%%%%%%%%%%%%%%%%%%%%%

\input{introduction}

\section{Preliminaries}
\subsection{Free Groups and Whitehead Graphs}
\input{free}

\subsection{Quasi-trees}
\input{qtree}

\subsection{Geometric Models}\label{sec:model}
\input{bs}
\section{Proof of \fullref{theorem:SSS}}
\input{detour}

\section{Factors}\label{sec:factors}
\input{factor}
\section{Proof of \fullref{theorem:vfree}}
\input{generalizations}

%%%%%%%%%%%%%%%%%%%%%%%%%%%%%%%%%%%%%
% End of main body of article
%%%%%%%%%%%%%%%%%%%%%%%%%%%%%%%%%%%%%

%%%%%%%%%%%%%%%%%%%%%%%%%%%%%%%%%%%%
% Start of backmatter
%%%%%%%%%%%%%%%%%%%%%%%%%%%%%%%%%%%%
\bibliographystyle{hyperamsplain}
\bibliography{masterbib}
%%%%%%%%%%%%%%%%%%%%%%%%%%%%%%%%%%%%
% End of backmatter
%%%%%%%%%%%%%%%%%%%%%%%%%%%%%%%%%%%%
\end{document}

%% file: abstract.tex
%!TEX root = Main.tex
\begin{abstract}
We give a geometric proof of a well known theorem that describes splittings of a free group as an amalgamated product or HNN extension over the integers. The argument generalizes to give a similar description of splittings of a virtually free group over a virtually cyclic group.
\end{abstract}

%% file: introduction.tex
%!TEX root = Main.tex
\section{Introduction}
This paper describes one-edge splittings of free groups
over (infinite) cyclic subgroups. 
Conversely, it describes when two free groups can be
amalgamated along a cyclic subgroup to form a free
group, or when an HNN--extension of a free group along a
cyclic subgroup is free.

\begin{theorem}[Shenitzer, Stallings, Swarup]\label{theorem:SSS}
Let $A$ and $B$ be finitely generated free groups, and let
$C$ be a cyclic group.
\begin{itemize}
\item $A*_C B$ is free if and only if one of the injections
  of $C$ into $A$ and $B$ maps $C$ onto a free factor of the vertex group.
\item $A*_C$ is free if and only if, up to $A$-conjugation,
  the edge injections map $C$ into independent free factors of $A$, and
  one of them is onto its factor.
\end{itemize}
\end{theorem}

This theorem is well known.
The amalgamated product case is a theorem of Shenitzer
\cite{She55}. 
The HNN case follows from a theorem of Swarup \cite{Swa86}, who proves
a more general theorem for splittings of free groups over free
subgroups. 
Swarup
attributes the case of cyclic splittings to Stallings.
A published version of Stallings' proof appears later \cite{Sta88}. 
A simple topological proof appears in an
unpublished paper of Bestvina and Feighn \cite[Lemma 4.1]{BesFei93}. Generalized versions appear in work of Louder \cite{Lou06} and Diao
and Feighn \cite{DiaFei05}.

We prove the theorem geometrically by showing that if the conditions are not satisfied then
the ``obvious'' geometric model for the group is not a quasi-tree, so
the group is not even
virtually free.
This proof generalizes to virtually
free groups:

\begin{theorem}\label{theorem:vfree}
Let $A$ and $B$ be finitely generated virtually free groups, and let
$C$ be a virtually cyclic group.
\begin{itemize}
\item $A*_C B$ is virtually free if and only if one of the injections
  of $C$ into $A$ and $B$ maps $C$ onto a factor of the vertex group.
\item $A*_C$ is virtually free if and only if, up to $A$-conjugation,
  the edge injections map $C$ into independent factors of $A$, and
  one of them is onto its factor.
\end{itemize}
\end{theorem}

We call an infinite subgroup $H$ of a group $G$ a \emph{factor} if $H$ is a
vertex group in a graph of groups decomposition of $G$ with
finite edge groups, and we call two factors \emph{independent} if they are the vertex
groups in the same graph of groups decomposition of $G$
with finite edge groups.

\fullref{theorem:vfree} can also be derived from more general
machinery for hyperbolic-elliptic splittings, for example,
\cite[Theorem 7.2]{DiaFei05}. The proof given here is different.

The ``if'' direction of the theorem is easy.
In the HNN case it may be necessary to change the stable letter to
account for the $A$-conjugation, but this is an isomorphism.
Now, replace the appropriate vertex group by a graph of groups in
which $C$ is a factor.
The stabilizer of the $C$-edge is mapped isomorphically to the stabilizer of
one of its end vertices, so this edge can be collapsed to give a graph of virtually free groups with finite
edge groups, and such a group is virtually free.

The theorem says that reversing this edge collapse move is the only
way to create a virtually free group as an amalgam over a virtually
cyclic subgroup.

First we prove the torsion free case. 
There are three ideas:
\begin{enumerate}
\item 
If a maximal cyclic subgroup of a free group is  not a
  free factor then Whitehead graphs for a generator of the subgroup
  are highly connected, in a certain sense.
\item The Whitehead graphs for the images of the amalgamated subgroup
  in the vertex groups record the intersection patterns of edge spaces
  in Bass-Serre complexes corresponding to the splitting.
\item High connectivity in the Whitehead graphs allows us to build
  sequences of
  paths in the Bass-Serre complex that connect the endpoints at infinity of some edge space but
  avoid arbitrarily large balls.
Thus, there are distinct points in the boundary at infinity that lie
in the same end of the group, but this does not happen in virtually
free groups.
\end{enumerate}

The proof relies, of course, on the important fact that virtually free
groups can be characterized geometrically.
Otherwise it is quite easy, using standard constructions of Whitehead
graphs and Bass-Serre complexes and elementary arguments.

The proof in the general case is similar, with the Bass-Serre complexes
coarsened.
The main work is to characterize factors in virtually free groups and
show the analog of item (1), which is done in 
\fullref{sec:factors}.

\medskip 

Gilbert Levitt has pointed out that some virtually free groups do not
have any virtually cyclic factors according to our definition.
An example is $\mathbb{Z}/2\mathbb{Z}*\mathbb{Z}/3\mathbb{Z}$, see \fullref{ex}.
Consequently, no HNN extension
of $\mathbb{Z}/2\mathbb{Z}*\mathbb{Z}/3\mathbb{Z}$ over a virtually
cyclic group  is ever virtually
free, nor is any  amalgam of two copies of
$\mathbb{Z}/2\mathbb{Z}*\mathbb{Z}/3\mathbb{Z}$ over a virtually
cyclic group.
Conversely, no non-trivial splitting of a virtually free group over a
virtually cyclic subgroup ever has
$\mathbb{Z}/2\mathbb{Z}*\mathbb{Z}/3\mathbb{Z}$ as a vertex group.

%% file: free.tex
% free groups and Whitehead graphs

Let $F=F_n$ be a finite rank free group.
A \emph{multiword} $\multimot$ is a finite list of words in $F$.
A free generating set $\mathcal{B}=\{b_1,\dots, b_n\}$ is called a \emph{basis}.
A multiword $\multimot=\{w_1,\dots, w_k\}$ is \emph{basic} if there
exist elements $f_i\in F$ such that $\{\inv{f}_iw_if_i\}$ is a subset
of a basis.
An element is \emph{indivisible} if it is not a proper power of
another element.
Basic elements are often called \emph{primitive} in the literature.

Let $|g|_{\mathcal{B}}$ denote the word length of an element $g$ with
respect to the basis $\mathcal{B}$.
Let $|[g]|_\mathcal{B}$ denote the minimum $\mathcal{B}$-length of a
conjugate of $g$.

The Whitehead graph $\Wh_{\mathcal{B}}(*)\{w\}$ of an indivisible,
cyclically reduced word $w\in F$
with respect to a basis $\mathcal{B}$ is a graph with vertex set in bijection with the set
$\mathcal{B}\cup\inv{\mathcal{B}}$ of generators and their inverses.
An edge is added from vertex $\inv{x}$ to vertex $y$ for each
occurrence of $xy$ as a subword of $w$ written as a reduced cyclic
word in the letters $\mathcal{B}\cup\inv{\mathcal{B}}$.

We can similarly define a Whitehead graph for a finite list of words
$\multimot$. 
We will be interested in the conjugacy classes of maximal cyclic
subgroups containing the words of $\multimot$.
Thus, to define $\Wh_{\mathcal{B}}(*)\{\multimot\}$ we choose a minimal set of
indivisible, cyclically reduced words $\underline{v}=\{v_i\}$ so that
each $w_i\in\multimot$ is conjugate into some $\left<v_j\right>$.
Then add edges as above for each $v_j$.
The graph constructed is independent of the choice of the $v_j$'s.

Whitehead's Algorithm \cite{Whi36} produces a point in the $\Aut(F)$ orbit of $w$
of minimal $\mathcal{B}$--length. 
An equivalent formulation for multiwords is that it chooses a basis $\mathcal{B}$ with respect to which
$\Wh_{\mathcal{B}}(*)\{\multimot\}$ has the minimal number of edges.

The Whitehead graphs we deal with will not always be connected, so we
make the following definitions:
\begin{definition}
A \emph{cut point} of a graph is a point such that deleting it
creates more connected components. A \emph{cut vertex} is a vertex
that is a cut point.   
\end{definition}

\begin{definition}
  We say a graph \emph{has 2-connected components} if every connected
  component is 2--connected.
\end{definition}

A special case of Menger's Theorem \cite{Men27} says a graph
without cut points has \2cc.

The next lemmas are easy exercises with Whitehead's Algorithm:
\begin{lemma}
  A Whitehead graph with a cut vertex is not minimal.
\end{lemma}

\begin{lemma}
  A Whitehead graph with a valence one vertex labeled $x$ is not
  minimal unless $x$ and $\inv{x}$ are joined by an isolated edge.
\end{lemma}

\begin{lemma}\label{lemma:loop}
Every non-trivial component of a minimal Whitehead graph is either 2--connected or
an isolated edge joining a vertex to its inverse.
\end{lemma}

\begin{lemma}\label{lemma:onecomponentperword}
Each word in a multiword contributes edges to only one component of a
minimal Whitehead graph.
\end{lemma}

For a fixed basis $\mathcal{B}$, the Cayley graph of $F$ with respect
to $\mathcal{B}$ is a tree $\tree$. 
The Whitehead graph can be generalized
to a Whitehead graph $\Wh_\mathcal{B}(\X)\{w\}$ over a compact subtree
$\X$ of $\tree$.
The vertex set is indexed by the elements of $F$ that are adjacent to
$\X$ in $\tree$.
Vertices labeled $u$ and $v$ are connected by an edge for each
$w$--orbit of  $\inv{u}v$ as a subword of some power of $w$.
One way to imagine this is that there is some cyclic permutation $w'$ of $w$ so that  if you start from the vertex $u$ in
$\tree$ and follow
the edge path that repeatedly spells out the word $w'$, eventually
you arrive at the vertex $v$.
Thus, $\Wh_\mathcal{B}(\X)\{w\}$ records the ``line pattern'' that conjugates
of $\left<w\right>$ make as they pass through $\X$.

The classical Whitehead graph $\Wh_{\mathcal{B}}(*)\{\multimot\}$ is the
generalized Whitehead graph such that the subtree $\X$ is just the
identity vertex $*\in\tree$.

Manning \cite{Man09} shows that generalized Whitehead graphs can be
constructed from classical Whitehead graphs by a construction called
\emph{splicing}.
It is an easy observation that splicing
connected graphs with no cut vertices produces a connected graph with
no cut vertices. 
This observation gives us the following generalization of \fullref{lemma:loop}, which will be used
later as an inductive step in building detours:

\begin{lemma}\label{detour}
  If  $\Wh_{\mathcal{B}}(*)\{\multimot\}$ has \2cc then $\Wh_{\mathcal{B}}(\X)\{\multimot\}$ has \2cc for every
 compact subtree  $\X\subset\tree$.
\end{lemma}

%% file: qtree.tex
% Quasi-trees
The terms in this section are standard (see, for example, \cite{BriHae99}.)
The following theorem gathers together various characterizations of
virtually free groups:
\begin{theorem}[Geometric Characterization of Virtually Free Groups]
Let $G$ be a finitely generated group.  
Let $X$ be a proper geodesic metric space quasi-isometric to $G$. 
The following are equivalent:
  \begin{enumerate}
\item $G$ is virtually free: it has a finite index free subgroup.\label{item:fifree}
\item $G$ has a finite index normal free subgroup.\label{item:finfree}
\item $G$ decomposes as a graph of virtually free groups with
  finite edge groups.\label{item:vfreedecomposition}
\item $G$ decomposes as a graph of finite groups.\label{item:finitedecomposition}
  \item  $X$ is a quasi-tree: there is a simplicial tree $\Gamma$ and a $(\lambda,\epsilon)$--quasi-isometry $\phi\from
    X\to \Gamma$.\label{item:qtree}
\item (Bottleneck Property) There is a constant $\Delta>0$ so that for all $x$ and $y$ in
  $X$ there exists a midpoint $m$ such that
  $d(x,m)=d(y,m)=\frac{1}{2}d(x,y)$ and such that any path from $x$ to
  $y$ passes through $N_\Delta(m)$.\label{item:bottleneck}
\item (Bottleneck Property') For any $K\geq 1$ and any $C\geq 0$ there is a $\Delta'\geq 0$
  so that for any $x$ and $y$ in $X$, any $(K,C)$--quasi-geodesic
  segment $\gamma$ joining $x$ to $y$, and any continuous path $p$
  from $x$ to $y$, we have $\gamma\subset
  N_{\Delta'}(p)$.\label{item:paths}
\item $X$ is hyperbolic and the natural map from $\bdry X$ onto
 $\mathrm{Ends}(X)$ is a bijection.\label{item:ends}
  \end{enumerate}
\end{theorem}
\begin{proof}

  (\ref{item:finfree}) follows easily from (\ref{item:fifree}).

The equivalence of (\ref{item:fifree}) and
(\ref{item:finitedecomposition}) is a theorem of Karass, Pietrowski,
and Solitar \cite{KarPieSol73}, using Stallings' Theorem \cite{Sta68}.
Item (\ref{item:vfreedecomposition}) is a variant.

(\ref{item:finitedecomposition}) implies (\ref{item:qtree}) since $G$
acts properly discontinuously  and cocompactly on the Bass-Serre tree
of the graph of groups decomposition.

The Bottleneck Property is due to Manning, who shows \cite[Theorem 4.6]{Man05} the equivalence of
(\ref{item:qtree})  and (\ref{item:bottleneck}).

Condition (\ref{item:paths}) is different version of the bottleneck
property.
It is just a coarsening of the fact that for
any two points $x$ and $y$ in a simplicial tree there is a unique
geodesic $[x,y]$ joining them, and any path $p$ joining $x$ to $y$
necessarily contains $[x,y]$.

$(\ref{item:qtree})\implies(\ref{item:paths})$ is proven by pushing
$\gamma$ and $p$ forward to $\Gamma$ with $\phi$, applying this fact,
and then pulling back to $X$ using a quasi-isometry inverse of $\phi$.

$(\ref{item:paths})\implies(\ref{item:bottleneck})$ is proven by
taking a geodesic segment $\gamma$ joining $x$ to $y$ and taking $m$
to be the midpoint of $\gamma$.
(\ref{item:bottleneck}) follows with $\Delta=\Delta'(1,0)$.

If $X$ is a quasi-tree it is hyperbolic and has a well defined
boundary at infinity.
(\ref{item:paths}) shows that no two boundary points lie in the
same end, thus (\ref{item:qtree})
implies (\ref{item:ends}).

Finally, if $G$ is finite the theorem is trivially true, and if it is
infinite and 
(\ref{item:ends}) holds then $\bdry X$ and $\mathrm{Ends}(X)$ have
at least two points.
By Stallings' Theorem $G$ splits over a finite group, and by
Dunwoody's Accessibility Theorem \cite{Dun85} there is a graph of
groups decomposition of $G$ over finite groups so that all of the
vertex groups are either finite or one-ended.
A one-ended vertex group would violate condition (\ref{item:ends}),
though, so $G$ satisfies condition (\ref{item:finitedecomposition}).
\end{proof}

To show something is not a quasi-tree we will show that it is possible
to detour around some bottleneck point, violating condition (\ref{item:paths}). Formally:

\begin{corollary}\label{notaqtree}
  A geodesic metric space $X$ is not a quasi-tree if there exists a
  quasi-geodesic $\gamma\from \mathbb{R}\to X$ and an increasing
  sequence $(t_i)$ of positive integers such that $\gamma(-t_i)$ and
  $\gamma(t_i)$ can be connected by a path that does not enter $N_i(\gamma(0))$.
\end{corollary}

%% file: bs.tex
% Bass-Serre Complexes
In the torsion free case we build a Bass-Serre complex $X$ for $G=A*_C
B$ as follows (the HNN case is similar).
Let $\mathcal{K}_A$ be a rose with $\pi_1(\mathcal{K}_A)=A$, and
similarly let $\mathcal{K}_B$ be a rose for $B$.
Let $\mathcal{K}_C=S^1\times [0,1]$ be an annulus with
$\pi_1(\mathcal{K}_C)=C$.
Build a space $\mathcal{K}$ with $\pi_1(\mathcal{K})=G$ by gluing one
boundary component of $\mathcal{K}_C$ to $\mathcal{K}_A$ according to
the edge injection $C\hookrightarrow A$, and similarly glue the other
boundary component to $\mathcal{K}_B$.
Let $X=\widetilde{\mathcal{K}}$.
See Scott-Wall \cite{ScoWal79} and Mosher-Sageev-Whyte
\cite{MosSagWhy04} for details.

A \emph{vertex space} is a connected component in $X$ of the preimage
of  $\mathcal{K}_A$ or $\mathcal{K}_B$.
In our case these are copies of Cayley trees for $A$ and $B$.
An \emph{edge strip} is a connected component of the preimage of one
of the $\mathcal{K}_C$, a bi-infinite, width 1 strip.
The quotient map that collapses each vertex space to a point and each
edge strip to an interval gives a $G$-equivariant map from $X$ to the
Bass-Serre tree of the graph of groups decomposition of $G$, so we call
$X$ the \emph{Bass-Serre Complex}.

The edge strips glue onto the vertex spaces along conjugates of the
image $\left<w\right>$ of the edge inclusion.
Thus, the Whitehead graph for $w$, or for $\{w_1,w_2\}$ in the HNN
case, records the intersection pattern of edge strips in a vertex space.

We will refer to paths that remain within a single vertex space as
\emph{horizontal}, and paths that go directly across an edge space as \emph{vertical}.

In the presence of torsion we can use the same construction to build a Bass-Serre complex, but the vertex and edge spaces may not be so
nice.
However, in the proof we will only need the fact that $G$
and $X$ are quasi-isometric, not that $G$ acts nicely on $X$.
Thus, we can make a trade: we will build a ``nicer'' space $X'$
quasi-isometric to $G$, but sacrifice the $G$ action to do so.
To do this we will choose finite index normal free subgroups $A'$ and
$B'$ of $A$ and $B$, respectively.
Fix bases for each of these, and replace each $A$--vertex space in $X$ by a
copy of the Cayley tree for $A'$, and similarly for $B$.
Each edge strip of $X$ glues on to an $A$--vertex space and a
$B$--vertex space along coarsely well defined lines, and we can use
quasi-isometry inverses to the inclusion maps $A'\hookrightarrow A$
and $B'\hookrightarrow B$ to give lines in the $X'$ vertex spaces to
attach edge strips to (see \fullref{eq}).
The resulting space $X'$ is a coarse Bass-Serre complex (see
Mosher-Sageev-Whyte \cite[Section 2.6]{MosSagWhy04}).

%%% Local Variables: 
%%% mode: latex
%%% TeX-master: "Main"
%%% End: 

%% file: detour.tex
% detours
\subsection{Amalgamated Product Case}
First, consider $G=A*_{\left<w\right>} B$.
Choose a basis for the minimal free factor $\hat{w}_A$ of $A$ containing
$w$ such that $w$ has minimal length and
extend it arbitrarily to a basis $\mathcal{B}_A$ of $A$.
Let $\mathcal{K}_A$ be the rose with $\rank(A)$ petals in bijection with
$\mathcal{B}_A$.
Repeat the construction for $B$.

Metrize $\mathcal{K}_A$ so that the edges have length $|[w]|_{\mathcal{B}_B}$.
Metrize $\mathcal{K}_B$ so that the edges have length $|[w]|_{\mathcal{B}_A}$.
Let $\mathcal{K}_C$ be a height 1 right annulus with boundary circles of length
$|[w]|_{\mathcal{B}_A}\cdot|[w]|_{\mathcal{B}_B}$.
With these choices the vertex spaces and
edge strips are isometrically embedded in the corresponding
 Bass-Serre complex $X$.

Choose a basepoint $\gamma(0)$ in an $A$--vertex space $X_\alpha$.
Define a map $\gamma\from |[w]|_{\mathcal{B}_A}\mathbb{Z}\to X$ by $\gamma(
|[w]|_{\mathcal{B}_A}\cdot t)=w^t.\gamma(0)$, and
extend linearly to get a map from $\mathbb{R}$.
To satisfy \fullref{notaqtree} it suffices to take the sequence
$(t_i= |[w]|_{\mathcal{B}_A}\cdot i)$.
To see this, for each $i>0$ we construct a path $p_i$ joining
$\gamma(-t_i)$ to $\gamma(t_i)$ that stays outside 
$N_i(\gamma(0))$.

Fix any $i>0$. Take the 0-th approximation $q_{i0}$ to $p_i$ to be the subsegment of
$\gamma$ connecting $\gamma(-t_i)$ to $\gamma(t_i)$.
Of course, this goes through $N_i(\gamma(0))$.

We will inductively push out the approximations of $p_i$ until we
leave  $N_i(\gamma(0))$, thereby creating a detour.
Depending on $X$ we can push vertically or horizontally.

First, suppose that $w$ is divisible in $A$. 
In this case, for any line in $X_\alpha$ to which an edge strip
attaches, there are at least two edge strips attached.
Construct $q_{i1}$ from $q_{i0}$ by pushing the segment vertically across one of the edge
strips that it lies on the boundary of.
That is, replace the horizontal segment $q_{i0}$ along one boundary
of the edge strip by a path that goes vertically across the edge
strip, horizontally across the opposite side, and then vertically
back.

The vertical segments of $q_{i1}$ lie outside $N_i(\gamma(0))$. 
The horizontal segment may not, but it has at least moved distance
one farther away from $\gamma(0)$ than $q_{i0}$.
This new horizontal segment lies in a $B$-vertex space.
Now, if $w$ is also divisible in $B$ then there are at least two edge strips that attach to the line
we have just arrived on. 
Thus, we can push the horizontal
segment vertically  across an edge strip different from the edge strip that we
used in the previous step, so that the horizontal segment gets farther
from $\gamma(0)$. 
Continuing in this way, the vertical
segments always stay outside $N_i(\gamma(0))$, and after $i$ steps the
horizontal segment is also outside $N_i(\gamma(0))$.

If $w$ is indivisible we must also push horizontally.
Suppose $w$ is indivisible in $A$.
By \fullref{lemma:onecomponentperword}, $\Wh_{\mathcal{B}_A}(*)\{w\}$
has one non-trivial connected component.
If $\left<w\right>$ is not a factor then the non-trivial connected
component is not an isolated edge, so by \fullref{lemma:loop}
$\Wh_{\mathcal{B}_A}(*)\{w\}$ has \2cc.
Note that since $X_\alpha$ is isometrically embedded, $X_\alpha\cap N_i(\gamma(0))$ is just the $i$--ball $N_i^{X_\alpha}(\gamma(0))$ in
$X_\alpha$ in its own natural metric (the one lifted from $\mathcal{K}_A$).
The two vertices $u$ and $v$ in $\closure{N_i(0)}\cap \gamma$ are adjacent in
$\Wh_{\mathcal{B}_A}(N^{X_\alpha}_i(\gamma(0)))\{w\}$; they are connected by an
edge $e$ corresponding to a segment of $q_{i0}$.
By \fullref{detour},
$\Wh_{\mathcal{B}_A}(N^{X_\alpha}_i(\gamma(0)))\{w\}$ has \2cc, so  there is another path connecting $u$
and $v$, an edge path $e_1,\dots, e_k$ that
does not use the edge $e$.
Each edge $e_j$ corresponds to a geodesic segment in $X_\alpha$
joining vertices outside of $N_i^{X_\alpha}(\gamma(0))$.
Construct $\overset{\frown}{q}_{i0}$ from $q_{i0}$ by replacing the $e$--segment by the
segments coming from the alternate path in the generalized Whitehead graph.

Each of the new horizontal segments has endpoints $u'$ and $v'$ outside of
$N_i(\gamma(0))$.
Furthermore, each of these new segments has an edge strip attached along
it.
Construct $q_{i1}$ from $\overset{\frown}{q}_{i0}$ by pushing each
horizontal segment vertically across an edge
strip.
As in the previous case, the vertical segments of $q_{i1}$ stay outside
$N_i(\gamma(0))$, and the horizontal
segments move farther from $\gamma(0)$.

The new horizontal segments lie in $B$-vertex spaces.
We can continue the construction if it is possible to push each of
these segments vertically or horizontally without pushing back across
an edge strip that was already crossed.
Thus, we would like to know that
each of these segments is on the boundary of two edge strips or that $\Wh_{\mathcal{B}_B}(*)\{w\}$ has \2cc.
If $\left<w\right>$ is not a factor of $B$ then one of these is true.

Thus, if $\left<w\right>$ is
a factor in neither $A$ nor $B$ we can push $\gamma$ out of any
$N_i(\gamma(0))$, so $X$ is not a quasi-tree, so
$A*_{\left<w\right>}B$ is not free. (Not even virtually free.)

\subsection{HNN Extension Case}
Let $G=A*_C=\left<A,t\mid \inv{t}w_1t=w_2\right>$, where $w_1$ and
$w_2$ are words in $A$. The edge injections are the maps $C\stackrel{\cong}{\longrightarrow} \left<w_i\right>$.

If $w_1$ and $w_2$ are conjugate into a common
maximal cyclic subgroup then $G$ contains a Baumslag-Solitar subgroup,
so it is not hyperbolic, hence not free.
Otherwise the vertex spaces are quasi-isometrically embedded and we
may repeat the construction from the amalgamated product case.

Take $t_i$ large enough so that $d(\gamma(\pm
t_i),\gamma(0))\geq 2i$.
If there is an initial horizontal push, take the new set of vertices
to also lie outside $N_{2i}(\gamma(0))$.
A vertical segment from such a vertex may lead closer to $\gamma(0)$, but
stays outside $N_{2i-1}(\gamma(0))$. 
Make sure the next round of horizontal pushing gives vertices outside
of $N_{2i-1}(\gamma(0))$, so that the next vertical segments stay
outside $N_{2i-2}(\gamma(0))$, etc.
$N_i(\gamma(0))$ still reaches across at most $i-1$ edge strips, so at
the $i$--th stage all vertical and horizontal segments lie outside $N_i(\gamma(0))$.

If $w_1$ and $w_2$ are both divisible then we only need to push
vertically, as before, to avoid the bottleneck point, so $G$ is not virtually free.

Otherwise, choose a basis $\mathcal{B}$ so that the Whitehead graph for $\multimot=\{w_1,w_2\}$
is minimal. 
Recall that by definition
$\Wh_{\mathcal{B}}(*)\{\multimot\}=\Wh_{\mathcal{B}}(*)\{\underline{v}\}$
where $\underline{v}=\{v_1,v_2\}$ such that $v_1$ and $v_2$ 
are indivisible, cyclically reduced with respect to $\mathcal{B}$, and so that
there exists an $a_i\in A$ such that $w_i\in a_i\left<v_i\right>\inv{a_i}$.
We may assume $a_1$ is trivial.

There are two possibilities. Either
$\Wh_{\mathcal{B}}(*)\{\underline{v}\}$ has only one non-trivial
connected component, or it has distinct components corresponding to $v_1$
and $v_2$.
In the first case the component has more than one edge, so, by
\fullref{lemma:loop}, $\Wh_{\mathcal{B}}(*)\{\underline{v}\}$ has \2cc.

In the second case, for each $i$ either the component containing $v_i$
is 2--connected or it is an isolated edge and $v_i$ is basic.

Thus, we can repeat the construction to build a path avoiding the
bottleneck point, and $G$ is not
virtually free, unless for some $i$, say $i=2$, we have both:
\begin{itemize}
\item $w_2$ is indivisible, and
\item $v_2$ is basic and gives an isolated edge in $\Wh_{\mathcal{B}}(*)\{\underline{v}\}$.
\end{itemize}

Now, the second condition implies there is a splitting
$A=A'*\left<v_2\right>$ with $w_1\in \left<v_1\right>\subset A'$.
If $w_2$ is
indivisible then $w_2=a_2v_2\inv{a_2}$ (after possibly exchanging $v_2$
and $\inv{v_2}$), so
\[A =A'*\left<v_2\right>= A'*\inv{a_2}\left<w_2\right>a_2\]

Thus, $G$ is not free unless, up to $A$-conjugation, the edge
injections map $C$ into independent factors, and one of them is onto.

%% file: factor.tex
To prove the theorem with torsion we will need a
characterization of when an infinite subgroup is a factor.
Recall this means that the subgroup appears as a vertex group in a
graph of groups decomposition with finite edge groups.
We make use of some results about the boundaries of relatively
hyperbolic groups due to Bowditch \cite{Bow11} and Groves and Manning \cite{GroMan08}.

A collection of subgroups $\underline{H}=\{H_1,\dots, H_k\}$ is an \emph{almost
  malnormal collection} if $|gH_i\inv{g}\cap H_j|=\infty$ implies
$i=j$ and $g\in H_i$.

If $G$ is a finitely generated hyperbolic group and $\underline{H}$ is
an almost malnormal collection of infinite, finitely generated, quasi-convex subgroups,
then $G$ is hyperbolic relative to $\underline{H}$ \cite[Theorem
7.11]{Bow11}.
There is a relatively hyperbolic boundary of $(G,\underline{H})$ that
we will denote $\mathcal{D}_{\underline{H}}$. 
This can be seen as the boundary of the ``cusped space'' obtained
from $G$ by hanging a horoball off each conjugate of each of the $H_i$'s
\cite{GroMan08}. 
The effect of this is to collapse the embedded image of each boundary
of a conjugate of an $H_i$ to a point. 
Thus, $\mathcal{D}_{\underline{H}}$ is the \emph{decomposition space}
that has one point for each distinct conjugate of each $H_i$ and one
point for each boundary point of $G$ that is not a boundary point of
some conjugate of an $H_i$.

We say that $G$ splits relative to $\underline{H}$ if there is a
splitting of $G$ so that each  $H_i$ is conjugate into
a vertex group of the splitting.
It is easy to see that corresponding to each edge in the Bass-Serre
tree of a splitting of $G$ over a finite group relative to
$\underline{H}$ there is a pair of complementary
nonempty clopen sets of $\mathcal{D}_H$.
Moreover, there is an analogue \cite[Proposition 10.1]{Bow11} of
Stallings' Theorem: $G$ splits over a finite group relative to
$\underline{H}$ if and only if $\mathcal{D}_{\underline{H}}$ is not
connected.

\begin{proposition}
Let $H$ be an infinite subgroup of a finitely generated hyperbolic
group $G$.
 Then $H$ is a factor of $G$ if and only if $H$ is finitely generated, quasi-convex,
 almost malnormal, and the connected component of $\mathcal{D}_H$
 containing the image of $\bdry H$ is a single point.
\end{proposition}
\begin{proof}
The ``only if'' direction is easy.
For the converse, suppose $H$ is not a proper factor of $G$.
We will show $H=G$.

$H$ is infinite, so there is a unique minimal factor containing it.
A factor of a factor is a factor, since finite groups act elliptically
on any tree, so we may assume $H$ is not contained in a proper factor
of $G$.
This means that $G$ does not split relative to $H$, so $\mathcal{D}_H$
is connected. Since the component containing the image of $\bdry H$ is
a single point, all of $\mathcal{D}_H$ is a single point.
This means the inclusion of $H$ into $G$ induces a homeomorphism between
$\bdry H$ and $\bdry G$.
Since $H$ is finitely generated this implies that $H$ is a finite index
subgroup of $G$. 
However, $H$ is almost malnormal, so the index must be one.
\end{proof}

\begin{corollary}\label{proposition:factorimpliestotdisc}
Let $H$ be an infinite subgroup of a finitely generated virtually free
group $G$. 
Then $H$ is a factor of $G$ if and only if $H$ is finitely generated, almost malnormal, and $\mathcal{D}_H$ is totally disconnected.  
\end{corollary}
\begin{proof}
 Since $G$ is virtually free, $H$ is a factor if and only if $G$ has a graph of groups
 decomposition such that $H$ is a vertex group and all other local
 groups are finite.
The components of $\mathcal{D}_H$ in this case are singletons for each
conjugate of $H$ and each end of the Bass-Serre tree of the splitting.
\end{proof}

\begin{proposition}\label{proposition:independentfactors}
  Let $\underline{H}=\{H_1,H_2\}$ be an almost malnormal collection of
  infinite, finitely generated, quasi-convex subgroups of a hyperbolic
  group $G$.
Up to conjugation, $H_1$ and $H_2$ are contained in independent
factors of $G$ if and only if the component of
$\mathcal{D}_{\underline{H}}$ containing the image of $\bdry H_1$ does
not contain the image of the boundary of any conjugate of $H_2$.
\end{proposition}
\begin{proof}
  The ``only if'' direction is easy. For the converse, for each $i$
  let $\hat{H}_i$ be the smallest factor containing $H_i$.
The image of $\bdry\hat{H}_i$ is a connected component of
$\mathcal{D}_{\underline{H}}$.
The hypothesis then implies that $\{\hat{H}_1,\hat{H}_2\}$ is an
almost malnormal collection whose decomposition space is not
connected.
Pass to a maximal graph of groups splitting of $G$ over finite groups
relative to $\{\hat{H}_1,\hat{H}_2\}$.
The hypothesis implies that $\hat{H}_1$ and $\hat{H}_2$, hence $H_1$
and $H_2$, are conjugate into different vertex groups of this splitting.
\end{proof}

\subsection{Virtually Cyclic Factors of Virtually Free Groups}
In this section let $H$ be an almost malnormal, virtually cyclic subgroup of a finitely
generated virtually free group $G$, and let $F$ be a finite index normal
free subgroup of $G$.
We relate connectivity of
$\mathcal{D}_H$ to connectivity of Whitehead graphs.

Choose representatives $g_i$ so that $G=\coprod
Fg_i$.
The map $\inv{\iota}\from G\to F\from fg_i\mapsto f$ is a
quasi-isometry inverse to the inclusion $\iota\from F \hookrightarrow G$.
Let $\left<w\right>=F\cap H$. 
This is a maximal cyclic subgroup of $F$ since $H$ is almost malnormal.
Let $d_i$ be double coset
representatives of $F\backslash G/\left<w\right>$.
Let $\multimot=\{d_iw\inv{d_i}\}$.

\begin{definition}
  The multiword $\multimot=\{d_iw\inv{d_i}\}$ above is \emph{a lift of
    $H$ to $F$}.
\end{definition}

For every $g\in G$ there exist $f\in F$, $g_i$, $d_j$, and $f'\in F$
such that $g=fg_i\in f'd_j\left<w\right>$.
Thus, $\inv{\iota}$ coarsely takes each $G$--conjugate of $H$ to an
$F$--conjugate of some $\left<d_jw\inv{d_j}\right>$:
\begin{equation}
\inv{\iota}(gH\inv{g})=\inv{\iota}(fg_iH\inv{g_i}\inv{f})\stackrel{c}{=}\inv{\iota}(fg_i\left<w\right>\inv{g_i}\inv{f})=\inv{\iota}(f'd_j\left<w\right>\inv{d_j}\inv{f'})=f'\left<d_jw\inv{d_j}\right>\inv{f'}\label{eq}
\end{equation}

(The second equivalence is coarsely true.)
It follows that $\mathcal{D}_H$ is homeomorphic to the decomposition
space of the boundary of $F$ obtained from the almost malnormal
collection $\{\left<d_iw\inv{d_i}\right>\}$, which we shall denote by
$\mathcal{D}_{\multimot}$.
Thus, to decide if $\mathcal{D}_H$ is totally disconnected we can lift
the problem to $F$ and consider $\mathcal{D}_{\multimot}$.

\begin{remark}
 We took $F$ to be normal so that $\multimot$ would have a nice form,
 but lifting to any finite index subgroup gives a homeomorphism of
 decomposition spaces.
\end{remark}

\begin{lemma}\label{lemma:basic}
  Let $\multimot$ be a multiword in a free group whose elements generate distinct
  conjugacy classes of maximal cyclic
  subgroups. The following are equivalent:
  \begin{enumerate}
\item $\multimot$ is basic.\label{item:basic}
\item Some minimal Whitehead graph for $\multimot$ consists of
  isolated edges.\label{item:somewh}
\item Every minimal Whitehead graph for $\multimot$ consists of
  isolated edges.\label{item:everywh}
  \item $\mathcal{D}_{\multimot}$ is totally disconnected.\label{item:totdisc}
  \end{enumerate}
\end{lemma}
\begin{proof}
Using Whitehead's Algorithm, the equivalence of (\ref{item:basic}), (\ref{item:somewh}), and
(\ref{item:everywh}) is easy.

  If $\multimot$ is basic we may take a graph of groups decomposition
  of $F$ with finite edge groups
  whose  cyclic vertex groups are generated by conjugates of the words in
  $\multimot$.
The same argument as \fullref{proposition:factorimpliestotdisc} shows
that $\mathcal{D}_{\multimot}$ is totally disconnected.
Thus, (\ref{item:basic}) implies (\ref{item:totdisc}).

Suppose (\ref{item:everywh}) is false, so that some minimal Whitehead
graph has a component containing more than one edge.
By passing to a free factor we may
assume that the Whitehead graph is connected.
Since it has \2cc, this is not a rank one factor.
It follows (see, for example, \cite[Theorem 4.1]{CasMac11}) that the decomposition space of the factor is
connected and not a single point, so $\mathcal{D}_{\multimot}$ is not
totally disconnected.
Thus, (\ref{item:totdisc}) implies (\ref{item:everywh}).
\end{proof}

\begin{lemma}\label{corollary:detour}
Let $\multimot=\{d_iw\inv{d_i}\}$ as above be
a lift of $H$ to $F$.
The following are equivalent:
\begin{enumerate}
\item $H$ is a factor of $G$.
\item $\multimot\subset F$ is basic.
\item Every minimal Whitehead graph of $\multimot$ consists of
  isolated edges.
\item Some minimal Whitehead graph of $\multimot$ contains an isolated edge.
\end{enumerate}

The alternative is that every minimal Whitehead graph of $\multimot$ has 2-connected components.
\end{lemma}
\begin{proof}
The alternative follows from \fullref{lemma:loop}.

Suppose some minimal Whitehead graph for $\multimot$ contains an
isolated edge.
Such an isolated edge would mean that for some
$i$ the point $(d_iw\inv{d}_i)^\infty$ is an
isolated point in $\mathcal{D}_{\multimot}$.
Since $\Aut(F)$ acts transitively on $\multimot$ and by homeomorphisms
on $\mathcal{D}_{\multimot}$, this would imply that $\mathcal{D}_{\multimot}$ is
totally disconnected.
By \fullref{lemma:basic}, this is equivalent to $\multimot$ being
basic and also to every minimal Whitehead graph consisting entirely of isolated
edges.
Furthermore, $\mathcal{D}_H$ and $\mathcal{D}_{\multimot}$ are
homeomorphic, and \fullref{proposition:factorimpliestotdisc} says that $H$ is a factor
if and only if $\mathcal{D}_H$ is totally disconnected.
\end{proof}

\begin{example}\label{ex}
  $G=\mathbb{Z}/2\mathbb{Z}*\mathbb{Z}/3\mathbb{Z}$ has no virtually
  cyclic factors.
\end{example}
\begin{proof}
Let  $G=\left<r,s\mid r^3=s^2=1\right>$.
There is a rank 2 normal free subgroup
$F=\left<srsr^2,sr^2sr\right>$, and
$G/F=\left<[sr]\right>=\mathbb{Z}/6\mathbb{Z}$.
The action of $sr$ on the abelianization of $F$
has orbits of size 3 on lines through the origin.
Thus, the words in the lift of any virtually cyclic group $H$ to $F$ are not contained in less than three distinct conjugacy
classes of maximal cyclic subgroups.
A basic multiword in $\mathbb{F}_2$ has words in at most two conjugacy
classes of maximal cyclic subgroup, so, by the previous lemma,  $H$ is not a factor.
\end{proof}

%%% Local Variables: 
%%% mode: latex
%%% TeX-master: "Main"
%%% End: 

%% file: generalizations.tex
% generalizations

Let $G=A*_CB$ be an amalgamated product of virtually free groups
over a virtually cyclic group.

$\Comm_A(C)=\{a\in A\mid aC\inv{a}\cap C \text{ is finite index in
  both $C$ and $aC\inv{a}$}\}$ is the \emph{commensurator} of $C$ in
$A$. 
A theorem of Kapovich and Short \cite{KapSho96} says that an infinite,
quasi-convex subgroup of a hyperbolic group has finite index in its commensurator.
Since $C$ is virtually cyclic,  so is $\Comm_A(C)$, and $\Comm_A(C)=\{a\in A\mid
|aC\inv{a}\cap C|=\infty\}$.
Thus, $\Comm_A(C)$ is the smallest almost malnormal subgroup of $A$
containing $C$.

Choose a finite index normal free subgroup $A'$ of $A$.
Let $\left<w\right>=A'\cap \Comm_A(C)$, and let $\multimot_A=\{d_iw\inv{d_i}\}$ be a lift of
$\Comm_A(C)$ to $A'$. 
Choose a basis for $A'$ with respect to which
$\multimot_A$ is Whitehead minimal.
After making similar choices for $B$, let $X'$ be the coarse
Bass-Serre complex for $G$ described in \fullref{sec:model}.

The number of edge strips attaching to a given conjugate of a $\left<d_iw\inv{d_i}\right>$ in $A'$ is equal
to the index of $C$ in $\Comm_A(C)$.

$X'$ is a tree of trees glued together along bi-infinite, width 1 edge
strips just as in the torsion free case, and we repeat the previous
argument to show
that $X'$ is not a quasi-tree if, for each line in $A'$ and $B'$ to
which an edge strip attaches, either
\begin{itemize}
\item there is a second edge strip attached to that same line, or
\item we can follow different edge strips to detour around an
  arbitrarily large ball centered on that line.
\end{itemize}

Now suppose $C$ is not a factor of $A$. 
It could be that $C$ is not almost malnormal in $A$, in which case the first
condition above is satisfied for $A$.
If $C$ is almost malnormal and not a factor of $A$ then by
\fullref{corollary:detour} every minimal Whitehead graph for a lift
of $C=\Comm_A(C)$ to $A'$ has \2cc.
This gives us the second condition.

Thus, if $C$ is a factor of neither $A$ nor $B$ then $X'$ is not a
quasi-tree, so $A*_CB$ is not virtually free.

\medskip

The $G=A*_C$ case follows by making similar adjustments to the torsion
free HNN case.
The interesting case is when the images
$C_1$ and $C_2$ of
$C$ in $A$ form an almost malnormal collection.
\fullref{proposition:independentfactors} shows that if 
$C_1$ and $C_2$ are not, up to conjugation, contained in independent
factors, then the images of $\bdry C_1$ and some $\bdry gC_2\inv{g}$ are contained
in a common component of $\mathcal{D}_{\{C_1,C_2\}}$.
Since $\{C_1,C_2\}$ is an almost malnormal collection, this component
is not a singleton, so $\mathcal{D}_{\{C_1,C_2\}}$ is not totally disconnected.
It follows that a minimal Whitehead graph for a lift of $\{C_1,C_2\}$ to a finite index normal subgroup of
$A$ will have 2-connected components, so $G$ is not virtually free.

\vspace{-1em}